\documentclass[12pt]{amsart}

\usepackage{latexsym}
\usepackage{amsmath}
\usepackage{amsfonts}
\usepackage{amssymb}
\usepackage{geometry} 
\usepackage{graphicx,color}
\usepackage{epstopdf}
\usepackage{caption}
\usepackage{diagbox}
\usepackage{subfig}
\usepackage{bm}
\usepackage{algorithm}
\usepackage{algorithmic}
\usepackage{tikz}
\graphicspath{{figures/}}
\allowdisplaybreaks[4]

\newtheorem{theorem}{Theorem}[section]
\newtheorem{lemma}[theorem]{Lemma}

\theoremstyle{definition}

\newtheorem{example}[theorem]{Example}

\theoremstyle{remark}
\newtheorem{remark}[theorem]{Remark}

\numberwithin{equation}{section}

\newcommand{\Lt}{L_2(\Omega)}
\def\bz{\bm{\zeta}}
\def\asiph{a^{\text{sip}}_h}
\def\aarh{a^{\text{ar}}_h}

\def\bn{\mathbf{n}}

\def\Ho{H^1_0(\Omega)}
\def\O{\Omega}

\def\LT{{L_2(\O)}}

\def\cB{\mathcal{B}}
\def\cE{\mathcal{E}}
\def\cT{\mathcal{T}}
\def\HB{{H^1_\beta(\O)}}
\def\HBh{{H^1_\beta(\O;\mathcal{T}_h)}}
\def\HO{{H^1(\O)}}
\def\H2{{H^2(\O)}}

\def\ES{\Ho\times\Ho}
\def\FS{V_h\times V_h}
\def\FK{V_k\times V_k}
\def\FKm{V_{k-1}\times V_{k-1}}

\def\fC{\mathfrak{C}}
\def\fB{\mathfrak{B}}

\newcommand{\trinorm}[1]{%
  |\mkern-1.5mu|\mkern-1.5mu|
   #1
  |\mkern-1.5mu|\mkern-1.5mu|
}
\DeclareMathOperator*{\argmin}{argmin}

\begin{document}

\title[Robust DG-MG for An Elliptic Optimal Control Problem]{Robust Multigrid Methods for discontinuous Galerkin Discretizations of An Elliptic Optimal Control Problem}

\author{Sijing Liu}
\address{Department of Mathematics\\
University Of Connecticut\\
Storrs, CT\\
USA}
\email{sijing.liu@uconn.edu}

\date{\today}

\begin{abstract}
   We consider discontinuous Galerkin methods for an elliptic distributed optimal control problem and we propose multigrid methods to solve the discretized system. We prove that the $W$-cycle algorithm is uniformly convergent in the energy norm and is robust with respect to a regularization parameter on convex domains. Numerical results are shown for both $W$-cycle and $V$-cycle algorithms.
\end{abstract}

\maketitle

\section{Introduction}

In this paper we consider the following elliptic optimal control problem.
Let $\Omega$ be a bounded convex polygonal domain in $\mathbb{R}^n$ ($n=2, 3$), $y_d\in \LT$ and $\beta$ be a positive constant. Find
\begin{equation}\label{optcon}
(\bar{y},\bar{u})=\argmin_{(y,u)\in K}\left [ \frac{1}{2}\|y-y_d\|^2_{\LT}+\frac{\beta}{2}\|u\|^2_{\LT}\right],
\end{equation} 
where $(y,u)$ belongs to $K\subset H^1_0(\Omega)\times \LT$ if and only if 
\begin{equation}\label{eq:stateeq}
a(y,v)=\int_{\Omega}uv \ dx \quad \forall v\in H^1_0(\Omega).
\end{equation}
Here the bilinear form $a(\cdot,\cdot)$ is defined as
\begin{equation}\label{eq:abilinear}
  a(y,v)=\int_{\Omega} \nabla y\cdot \nabla v\ dx+\int_{\Omega} (\bm{\zeta}\cdot\nabla y) v\ dx+\int_{\Omega} \gamma yv\ dx,
\end{equation}
where vector field $\bm{\zeta}\in [W^{1,\infty}(\Omega)]^2$ and the function $\gamma\in L_{\infty}(\Omega)$ is nonnegative. If $\bm{\zeta}\ne0$ then the constraint \eqref{eq:stateeq} is the weak form of a general second order PDE with an advective/convective term. We assume
\begin{equation}\label{eq:advassump}
    \gamma-\frac12\nabla\cdot\bz\ge\gamma_0>0
\end{equation}
such that the problem \eqref{eq:stateeq} is well-posed. 

\begin{remark}
    Throughout the paper we will follow the standard notation for differential operators, function spaces and norms that can be found for example in \cite{BS,Ciarlet}.
\end{remark}

It is well-known that \cite{Lions, Tro} the solution of \eqref{optcon}-\eqref{eq:stateeq} is characterized by
\begin{subequations}\label{eq:sp}
\begin{alignat}{3}
a(q,\bar{p})&=(\bar{y}-y_d,q)_\LT \quad &&\forall q\in H^1_0(\Omega),\\
\bar{p}+\beta\bar{u}&=0,\\
a(\bar{y},z)&=(\bar{u},z)_\LT  \quad &&\forall z\in H^1_0(\Omega),
\end{alignat}
\end{subequations}
where $\bar{p}$ is the adjoint state.
After eliminating $\bar{u}$, we arrive at the saddle point problem
\begin{subequations}\label{eq:osp}
\begin{alignat}{2}
a(q,\bar{p})-(\bar{y},q)_\LT&=-(y_d,q)_\LT \quad &&\forall q\in H^1_0(\Omega),\\
-(\bar{p},z)_\LT-\beta a(\bar{y},z)&=0  \quad &&\forall z \in H^1_0(\Omega).
\end{alignat}
\end{subequations}

Notice that $\beta$ is only in the second equation in \eqref{eq:osp}. In order to make the equations more balanced we perform a change of variables.
Let
\begin{equation}
\bar{p}=\beta^{\frac{1}{4}}\tilde{p}
\quad 
\text{and} 
\quad 
\bar{y}=\beta^{-\frac{1}{4}}\tilde{y}
\end{equation}
then the system \eqref{eq:osp} becomes 
\begin{subequations}\label{eq:sp1}
\begin{alignat}{2}
\beta^{\frac{1}{2}}a(q,\tilde{p})-(\tilde{y},q)_\LT&=-\beta^{\frac{1}{4}}(y_d,q)_\LT\quad &&\forall q\in H^1_0(\Omega),\\
-(\tilde{p},z)_\LT-\beta^{\frac{1}{2}}a(\tilde{y},z)&=0 \quad &&\forall z\in H^1_0(\Omega).
\end{alignat}
\end{subequations}

In this paper we employ a discontinuous Galerkin (DG) method \cite{arnold1982interior,riviere2008discontinuous,arnold2002unified} to discretize the saddle point problem $\eqref{eq:sp1}$. There are several advantages to use DG methods, for example, DG methods are more flexible regarding choices of meshes and more suitable for convection-dominated problems. It is well-known that DG methods can capture sharp gradients in the solutions such that spurious oscillations can be avoided. Our goal is to design multigrid methods for solving the discretized system resulting from the DG discretization that are robust with respect to the regularization parameter $\beta$. 

Multigrid methods for \eqref{eq:sp1} based on continuous Galerkin methods are intensively studied in the literature, for example, in \cite{Borzi,brenner2020multigrid,takacs2013convergence,Scho,taasan1991one} and the references therein. In \cite{brenner2020multigrid}, based on the approaches in \cite{BLS,BOS,brenner2017multigrid}, the authors developed multigrid methods for \eqref{eq:sp1} using a continuous $P_1$ finite element method. Besides the robustness of the multigrid methods with respect to $\beta$, the estimates in \cite{brenner2020multigrid} are established in a natural energy norm and the multigrid methods have a standard $\mathcal{O}(m^{-1})$ performance where $m$ is the number of smoothing steps. In this paper, we extend the results in \cite{brenner2020multigrid} to DG methods where the diffusion term in \eqref{eq:abilinear} is discretized by a symmetric interior penalty (SIP) method and the convection term \eqref{eq:abilinear} is discretized by an unstabilized\textbackslash centered-fluxes DG method \cite{di2010discrete,brezzi2004discontinuous}.
Multigrid methods based on DG methods are investigated in \cite{brenner2005convergence,gopalakrishnan2003multilevel,antonietti2015multigrid,brenner2009multigrid,brenner2011multigrid,hong2016robust,kanschat2015multigrid} and the references therein. However, not much work has been done towards the multigrid methods based on DG discretization for optimal control problems with provable results. The general idea in this paper is to construct a block-diagonal preconditioner and convert the saddle point problem \eqref{eq:sp1} into an equivalent symmetric positive definite (SPD) problem using the preconditioner. Therefore well-known multigrid theories for SPD system \cite{Hackbusch,bramble1993multigrid} can be utilized. The preconditioner requires solving a reaction-diffusion equation (approximately) based on a SIP discretization. This, however, does not affect the overall optimal computational complexity of our multigrid methods since the preconditioner itself can be constructed by multigrid methods (cf. \cite{brenner2005convergence}).

The rest of the paper is organized as follows. In Section \ref{sec:contprob}, we gather some known results regarding the continuous problem. In Section \ref{sec:dgfem}, we discretize the optimal control problem with a DG method and establish concrete error estimates. A crucial block-diagonal preconditioner is introduced and the 
multigrid methods for the discretized system are described in Section \ref{sec:multigrid}. In Sections \ref{sec:smapprox} and \ref{sec:convanalysis}, we establish the main theorem which is based on the smoothing and approximation properties of the multigrid methods. Finally, we provide some numerical results in Section \ref{sec:numerics} and end with some concluding remarks in Section \ref{sec:conclude}. Some technical proofs are provided in Appendices \ref{apdix:pfrh} and \ref{apdix:duality}.

Throughout this paper, we use $C$
 (with or without subscripts) to denote a generic positive
 constant that is independent of any mesh
 parameters and the regularization parameter $\beta$.
  Also to avoid the proliferation of constants, we use the
   notation $A\lesssim B$ (or $A\gtrsim B$) to
  represent $A\leq \text{(constant)}B$. The notation $A\approx B$ is equivalent to
  $A\lesssim B$ and $B\lesssim A$. Note that we do not consider convection-dominated case in this paper, hence the constants might depend on $\bm{\zeta}$ and $\gamma$.

\section{Continuous Problem}\label{sec:contprob}

We rewrite \eqref{eq:sp1} in a concise form
\begin{equation}\label{eq:conciseform}
\mathcal{B}((\tilde{p},\tilde{y}),(q,z))=-\beta^{\frac{1}{4}}(y_d,q)_\LT \quad \quad \forall (q,z)\in H^1_0(\Omega)\times H^1_0(\Omega),
\end{equation}
where
\begin{equation}\label{bilinear}
\mathcal{B}((p,y),(q,z))=\beta^{\frac{1}{2}}a(q,p)-(y,q)_\LT-(p,z)_\LT-\beta^{\frac{1}{2}}a(y,z).
\end{equation}

Let $\|p\|_{H^1_{\beta}(\Omega)}$ be defined by
\begin{equation}
\|p\|^2_{H^1_{\beta}(\Omega)}=\beta^{\frac{1}{2}}|p|^2_{H^1(\Omega)}+\|p\|^2_{\LT}.
\end{equation}
We have the following lemmas (cf. \cite{brenner2020multigrid}) regarding the bilinear form $\cB$ with respect to the norm $\|\cdot\|_\HB$.
\begin{lemma}\label{lem:stability}
We have
\begin{equation}\label{eq:cont}
\|p\|_{H^1_{\beta}(\Omega)}+\|y\|_{H^1_{\beta}(\Omega)}\approx\sup_{(q,z)\in H^1_0(\Omega)\times H^1_0(\Omega)}\frac{\mathcal{B}((p,y),(q,z))}{\|q\|_{H^1_{\beta}(\Omega)}+\|z\|_{H^1_{\beta}(\Omega)}}
\end{equation}
for any $(p,y)\in H^1_0(\Omega)\times H^1_0(\Omega).$
\end{lemma}

\begin{remark}
Lemma \ref{lem:stability} guarantee the well-posedness of \eqref{eq:sp1} by the standard theory \cite{Bab,Brezzi}.
\end{remark}

\begin{remark}
    We also have the similar stability estimate
    \begin{equation}\label{eq:contdual}
\|p\|_{H^1_{\beta}(\Omega)}+\|y\|_{H^1_{\beta}(\Omega)}\approx\sup_{(q,z)\in H^1_0(\Omega)\times H^1_0(\Omega)}\frac{\mathcal{B}((q,z),(p,y))}{\|q\|_{H^1_{\beta}(\Omega)}+\|z\|_{H^1_{\beta}(\Omega)}}
\end{equation}
for any $(p,y)\in H^1_0(\Omega)\times H^1_0(\Omega).$
\end{remark}

We also need the following regularity results (cf. \cite{brenner2020multigrid}) on convex domains.
Let $(p,y)\in H^1_0(\Omega)\times H^1_0(\Omega)$ satisfies
\begin{equation}\label{eq:regu}
\mathcal{B}((p,y),(q,z))=(f,q)_\LT+(g,z)_\LT\quad \quad \forall (q,z)\in H^1_0(\Omega)\times H^1_0(\Omega),
\end{equation}
where $(f,g)\in \LT\times \LT$ and $\mathcal{B}$ is defined in \eqref{bilinear}.

\begin{lemma}\label{h2}
The solution $(p,y)$ of \eqref{eq:regu} belongs to $H^2(\O)\times H^2(\O)$ and we have
\begin{equation}
\|\beta^{\frac{1}{2}}p\|_{H^2(\Omega)}+\|\beta^{\frac{1}{2}}y\|_{H^2(\Omega)}\le C_{\Omega}(\|f\|_{\LT}+\|g\|_{\LT}).
\end{equation}
\end{lemma}

\begin{remark}
    Lemma \ref{h2} is also valid for the following dual problem. Find $(p,y)\in H^1_0(\Omega)\times H^1_0(\Omega)$ such that
\begin{equation}\label{eq:regudual}
\mathcal{B}((q,z),(p,y))=(f,q)_\LT+(g,z)_\LT\quad \quad \forall (q,z)\in H^1_0(\Omega)\times H^1_0(\Omega).
\end{equation}
\end{remark}

\section{Discrete Problem}\label{sec:dgfem}

In this section we discretize the saddle point problem \eqref{eq:sp1} by a DG method \cite{arnold2002unified,arnold1982interior,brezzi2004discontinuous}. Let $\mathcal{T}_h$ be a shape regular simplicial triangulation of $\Omega$.
The diameter of $T\in\mathcal{T}_h$ is denoted by $h_T$ and $h=\max_{T\in\mathcal{T}_h}h_T$ is the mesh diameter. 
Let $\mathcal{E}_h=\mathcal{E}^b_h\cup\mathcal{E}^i_h$ where $\cE^i_h$ (resp. $\cE^b_h$) represents the set of interior edges (resp. boundary edges).

We further decompose the boundary edges $\cE^b_h$ into the inflow part $\cE^{b,-}_h$ and the outflow part $\cE^{b,+}_h$ which are defined as follows,
\begin{align}
    \cE^{b,-}_h&=\{e\in\cE^b_h: e\subset\{x\in\partial\O: \bz(x)\cdot\mathbf{n}(x)<0\}\},\\
    \cE^{b,+}_h&=\cE^b_h\setminus\cE^{b,-}_h.
\end{align}

For an edge $e\in \mathcal{E}^i_h$, let $h_e$ be the length of $e$. For each edge we associate a fixed unit normal $\mathbf{n}$. We denote by $T^+$ the element for which $\mathbf{n}$ is the outward normal, and $T^-$ the element for which $-\mathbf{n}$ is the outward normal. We define the discontinuous finite element space $V_h$ as 
\begin{equation}
    V_h=\{v\in\LT:v|_T\in\mathbb{P}_1(T)\quad\forall\ T\in\mathcal{T}_h\}.
\end{equation}
For $v\in V_h$ on an edge $e$, we define
\begin{equation}
    v^+=v|_{T^+}\quad\text{and}\quad v^-=v|_{T^-}.
\end{equation}
We define the jump and average for $v\in V_h$ on an edge $e$ as follows,
\begin{equation}
    [v]=v^+-v^-,\quad \{v\}=\frac{v^++v^-}{2}.
\end{equation}
For $e\in\mathcal{E}_h^b$ with $e\in\partial T$, we let
\begin{equation}
    [v]=\{v\}=v|_T.
\end{equation}
We also denote 
\begin{equation}
    (w,v)_e:=\int_e wv\ \!ds\quad\text{and}\quad(w,v)_T:=\int_T wv\ \!dx.
\end{equation}

\subsection{Discontinuous Galerkin methods}
The DG methods for \eqref{eq:conciseform} is to find $(\tilde{p}_h,\tilde{y}_h)\in V_h\times V_h$ such that
\begin{equation}\label{eq:dg}
\mathcal{B}_h((\tilde{p}_h,\tilde{y}_h),(q,z))=-\beta^{\frac{1}{4}}(y_d,q)_\LT \quad \quad \forall (q,z)\in V_h\times V_h,
\end{equation}
where
\begin{equation}\label{eq:dgbilinear}
\mathcal{B}_h((p,y),(q,z))=\beta^{\frac{1}{2}}a_h(q,p)-(y,q)_\LT-(p,z)_\LT-\beta^{\frac{1}{2}}a_h(y,z).
\end{equation}
The bilinear form $a_h(\cdot,\cdot)$ is defined by
\begin{equation}\label{eq:ahdef}
    a_h(u,v)=a_h^{\text{sip}}(u,v)+a^{\text{ar}}_h(u,v)\quad\forall u,v\in V_h,
\end{equation}
where
\begin{equation}\label{eq:dgbilinearsip}
\begin{aligned}
    a^{\text{sip}}_h(u,v)=&\sum_{T\in\mathcal{T}_h}(\nabla u, \nabla v)_T-\sum_{e\in\mathcal{E}_h}(\{\mathbf{n}\cdot\nabla u\},[v])_e
    -\sum_{e\in\mathcal{E}_h}(\{\mathbf{n}\cdot\nabla v\},[u])_e\\
    &+\sigma\sum_{e\in\mathcal{E}_h} h_e^{-1}([u],[v])_e
\end{aligned}
\end{equation}
is the bilinear form of the SIP method with sufficiently large penalty parameter $\sigma$ and the unstabilized DG scheme (cf. \cite{brezzi2004discontinuous,di2011mathematical}) for the advection-reaction term is defined as
\begin{equation}\label{eq:ardef}
\begin{aligned}
    a^{\text{ar}}_h(u,v)=\sum_{T\in\mathcal{T}_h}(\bm{\zeta}\cdot\nabla u+\gamma u, v)_T
     -\sum_{e\in\cE^i_h\cup\cE^{b,-}_h}(\bn\cdot\bm{\zeta}[u],\{v\})_e.
\end{aligned}
\end{equation}

\begin{remark}
   We do not consider convection-dominated case in this paper. Therefore the bilinear form $\aarh(\cdot,\cdot)$ does not contain any upwind stabilization terms. However, if we consider convection-dominated case, the well-known upwind schemes \cite{leykekhman2012local,ayuso2009discontinuous,brezzi2004discontinuous} could be utilized.
\end{remark}

It is also necessary to consider the general problem \eqref{eq:regu} and the dual problem \eqref{eq:regudual}. The DG methods for \eqref{eq:regu} is to find $(p_h,y_h)\in V_h\times V_h$ such that
\begin{equation}\label{eq:dgprob}
\mathcal{B}_h((p_h,y_h),(q,z))=(f,q)_\LT+(g,z)_\LT\quad \quad \forall (q,z)\in V_h\times V_h.
\end{equation}
Similarly, the DG methods for \eqref{eq:regudual} is to find $(p_h,y_h)\in V_h\times V_h$ such that
\begin{equation}\label{eq:dgprobdual}
\mathcal{B}_h((q,z),(p_h,y_h))=(f,q)_\LT+(g,z)_\LT\quad \quad \forall (q,z)\in V_h\times V_h.
\end{equation}

Let the norm $\|\cdot\|_{\HBh}$ be defined as
\begin{equation}\label{eq:hbhnorm}
  \|p\|^2_{\HBh}=\beta^{\frac{1}{2}}\|p\|^2_{1,h}+\|p\|^2_{\LT},
\end{equation}
where 
\begin{equation}\label{eq:hbhnorm1}
  \|p\|_{1,h}^2=\sum_{T\in\mathcal{T}_h}\|\nabla p\|^2_{L_2(T)}+\sum_{e\in\mathcal{E}_h}\frac{1}{h_e}\|[p]\|^2_{L_2(e)}+\sum_{e\in\mathcal{E}_h}h_e\|\{\mathbf{n}\cdot\nabla p\}\|_{L_2(e)}^2.
\end{equation}

Let $V=\Ho\cap H^2(\O)$. For the bilinear form $a_h(\cdot,\cdot)$, we have
        \begin{alignat}{3}
            a_h(w,v)&\le C\|w\|_{1,h}\|v\|_{1,h}&&\quad \forall w,v\in V+V_h,\label{eq:ahcont}\\  
            a_h(v,v)&\ge C\|v\|^2_{1,h}&&\quad\forall w,v\in V_h,\label{eq:ahcoer}
        \end{alignat}
for sufficiently large $\sigma$. A proof is provided in Appendix \ref{apdix:pfrh}. Note that the constants in \eqref{eq:ahcont}-\eqref{eq:ahcoer} might depend on $\bm{\zeta}$ and $\gamma$.

It follows from \eqref{eq:dgbilinear}, \eqref{eq:hbhnorm}, \eqref{eq:ahcont} and Cauchy-Schwarz inequality that
\begin{equation}\label{eq:bhcont}
\mathcal{B}_h((p,y),(q,z))\le (\|p\|^2_{\HBh}+\|y\|^2_{\HBh})^{\frac{1}{2}}(\|q\|^2_{\HBh}+\|z\|^2_{\HBh})^{\frac{1}{2}}
\end{equation}
for any $(p,y), (q,z)\in (V+V_h)\times (V+V_h).$

We also have, by \eqref{eq:dgbilinear}, \eqref{eq:ahcoer} and a direct calculation,
\begin{equation}\label{eq:bhcoer1}
\begin{aligned}
  \mathcal{B}_h&((p,y),(p-y,-y-p))\\
  &=\beta^\frac12 a_h(p,p)+(p,p)_\LT+\beta^\frac12 a_h(y,y)+(y,y)_\LT\\
  &\gtrsim\|p\|^2_{\HBh}+\|y\|^2_{\HBh}
\end{aligned}
\end{equation}
and
\begin{equation}\label{eq:bhcoer2}
  \|p-y\|^2_{\HBh}+\|-y-p\|^2_{\HBh}=2(\|p\|^2_{\HBh}+\|y\|^2_{\HBh})
\end{equation}
by the parallelogram law.
It follows from \eqref{eq:bhcont}-\eqref{eq:bhcoer2} that
\begin{equation}\label{eq:bhwposed}
\begin{aligned}
  \|p_h\|_{\HBh}&+\|y_h\|_{\HBh}\\
  &\approx\sup_{(q,z)\in V_h\times V_h}\frac{\mathcal{B}_h((p_h,y_h),(q,z))}{\|q\|_{\HBh}+\|z\|_{\HBh}}\quad\forall(p_h,y_h)\in V_h\times V_h.
  \end{aligned}
\end{equation}
Similarly, we have 
\begin{equation}\label{eq:bhwposeddual}
\begin{aligned}
  \|p_h\|_{\HBh}&+\|y_h\|_{\HBh}\\
  &\approx\sup_{(q,z)\in V_h\times V_h}\frac{\mathcal{B}_h((q,z),(p_h,y_h))}{\|q\|_{\HBh}+\|z\|_{\HBh}}\quad\forall(p_h,y_h)\in V_h\times V_h.
  \end{aligned}
\end{equation}
It follows immediately from $\eqref{eq:bhwposed}$ and $\eqref{eq:bhwposeddual}$ that the discrete problems \eqref{eq:dgprob} and \eqref{eq:dgprobdual} are uniquely solvable. We also need the following lemma which follows from a standard inverse estimate (cf. \cite{BS}) and trace inequalities (cf. \cite[Proposition 3.1]{ciarlet2013analysis}),

\begin{lemma}
    Assume $\mathcal{T}_h$ is a quasi-uniform triangulation of $\O$, we have the following inverse estimate,
    \begin{equation}\label{eq:dginverse}
        \|v\|_{1,h}\lesssim h^{-1}\|v\|_\LT\quad\forall v\in V_h.
    \end{equation}
\end{lemma}

\subsection{Interpolation operator $\Pi_h$}
We use the usual continuous nodal interpolant (which belongs to $V_h$) \cite{arnold2002unified,riviere2008discontinuous,BS} such that the following estimates hold.

\begin{lemma}\label{lem:interpo}
We have 
\begin{equation}\label{eq:interpolation}
  \|z-\Pi_hz\|_{\Lt}+h|z-\Pi_hz|_{\HO}\lesssim h^2|z|_{H^2(\O)}\quad\forall z\in H^2(\O)\cap\Ho.
\end{equation}
\end{lemma}

\begin{remark}\label{remark:inter1h}
    We also have (cf. \cite{BS})
    \begin{equation}\label{eq:inter1h}
        \|z-\Pi_hz\|_{1,h}\lesssim h|z|_{H^2(\O)}\quad\forall z\in H^2(\O)\cap\Ho.
    \end{equation}
\end{remark}

\subsection{Error Estimates}

It is well-known \cite{arnold2002unified,riviere2008discontinuous,brezzi2004discontinuous} that the DG method \eqref{eq:dg} is consistent. Hence, we have the following Galerkin orthogonality
\begin{equation}\label{eq:galortho}
  \cB_h((p-p_h,y-y_h),(q,z))=0\quad\forall(q,z)\in\FS.
\end{equation}

\begin{lemma}\label{lem:derr}
Let the functions $(p,y)$ (resp,. $(p_h,y_h)$) be the solutions of \eqref{eq:regu} or \eqref{eq:regudual} (resp,. \eqref{eq:dgprob} or \eqref{eq:dgprobdual}), we have
\begin{equation}\label{eq:hb1esti}
\begin{aligned}
    \|p-p_h\|_{\HBh}&+\|y-y_h\|_{\HBh}\\
    &\le C_{\Omega}({\beta}^{\frac{1}{2}}h^{-2}+1)^{\frac{1}{2}}\beta^{-\frac{1}{2}}h^2(\|f\|_{\LT}+\|g\|_{\LT}),
\end{aligned}
\end{equation}
\begin{equation}\label{eq:l2esti}
    \begin{aligned}
        \|p-p_h\|_{\LT}&+\|y-y_h\|_{\LT}\\
        &\le C_{\Omega}({\beta}^{\frac{1}{2}}h^{-2}+1)\beta^{-1}h^4(\|f\|_{\LT}+\|g\|_{\LT}).
    \end{aligned}
\end{equation}
\end{lemma}

\begin{proof}
We only establish the estimates involving \eqref{eq:regu} and \eqref{eq:dgprob}. 
Using the estimate \eqref{eq:bhwposed} and the relation \eqref{eq:galortho}, we have
\begin{equation}
\|p-p_h\|_{\HBh}+\|y-y_h\|_{\HBh}\lesssim\inf_{(q,z)\in V_h\times V_h}(\|p-q\|_{\HBh}+\|y-z\|_{\HBh}).
\end{equation}
By Lemma \ref{lem:interpo} and Lemma \ref{h2}, we have
\begin{equation}
\beta^{\frac{1}{2}}\|p-\Pi_hp\|^2_{1,h}\le C\beta^{\frac{1}{2}}h^2\|p\|_{H^2(\Omega)}^2\le \frac{C}{\beta^{\frac{1}{2}}}h^2(\|f\|_{\LT}+\|g\|_{\LT})^2
\end{equation}
 and 
 \begin{equation}
\|p-\Pi_hp\|^2_{\LT}\le C h^4\|p\|_{H^2(\Omega)}^2\le\frac{C}{\beta}h^4(\|f\|_{\LT}+\|g\|_{\LT})^2.
\end{equation}
Thus we obtain
\begin{equation}
\|p-\Pi_hp\|_{\HBh}^2\le C(\beta^{-\frac{1}{2}}h^2+\beta^{-1}h^4)(\|f\|_{\LT}+\|g\|_{\LT})^2.
\end{equation}
It is similar to estimate $\|y-\Pi_hy\|_{\HBh}^2$. Therefore we have the estimate \eqref{eq:hb1esti}.
The estimate \eqref{eq:l2esti} is established by a duality argument. Let $(\xi,\theta)$ satisfies 
\begin{subequations}\label{eq:duality}
\begin{alignat}{3}
&\beta^{\frac{1}{2}}(-\Delta\xi+\bz\cdot\nabla\xi+\gamma\xi)-\theta=p-p_h,\quad &\xi=0\quad\text{on}\quad\partial\O,\\
&-\xi-\beta^{\frac{1}{2}}(-\Delta\theta-\bz\cdot\nabla\theta+(\gamma-\nabla\cdot\bz)\theta)=y-y_h,\quad &\theta=0\quad\text{on}\quad\partial\O.
\end{alignat}
\end{subequations}
The weak form of \eqref{eq:duality} is to find $(\xi,\theta)\in H^1_0(\Omega)\times H^1_0(\Omega)$ such that 
\begin{equation}\label{eq:dualityweak}
    \cB((q,z),(\xi,\theta))=(q,p-p_h)_\LT+(z,y-y_h)_\LT\quad\forall(q,z)\in\ES.
\end{equation}
We can show that (cf. \cite{leykekhman2012investigation})
\begin{equation}\label{eq:dualityidentity}
    \|p-p_h\|^2_{\LT}+\|y-y_h\|^2_{\LT}=\cB_h((p-p_h,y-y_h),(\xi,\theta)).
\end{equation}
A proof is provided in Appendix \ref{apdix:duality}.

Then it follows from Cauchy-Schwarz inequality, \eqref{eq:galortho}, \eqref{eq:dualityidentity} and Lemma \ref{h2} (apply to \eqref{eq:dualityweak}) that
\begin{equation}\label{eq:dualityprop}
\begin{aligned}
&\|p-p_h\|^2_{\LT}+\|y-y_h\|^2_{\LT}\\
&=\cB_h((p-p_h,y-y_h),(\xi,\theta))\\
&=\cB_h((p-p_h,y-y_h),(\xi-\Pi_h\xi,\theta-\Pi_h\theta))\\
&\lesssim(\|p-p_h\|^2_{\HBh}+\|y-y_h\|^2_{\HBh})^\frac12\\
&\hspace{0.3cm}\times(\|\xi-\Pi_h\xi\|^2_{\HBh}+\|\theta-\Pi_h\theta\|^2_{\HBh})^\frac12\\
&\lesssim(\beta^{-\frac12}h^2+\beta^{-1}h^4)^\frac12(\|p-p_h\|^2_\LT+\|y-y_h\|^2_\LT)^\frac12\\
&\hspace{0.3cm}\times (\|p-p_h\|^2_{\HBh}+\|y-y_h\|^2_{\HBh})^\frac12.
\end{aligned}
\end{equation}
The estimate \eqref{eq:l2esti} then follows from \eqref{eq:dualityprop} and \eqref{eq:hb1esti}.
\end{proof}

\section{Multigrid Methods}\label{sec:multigrid}

In this section we introduce the multigrid methods for \eqref{eq:regu}. A crucial block-diagonal preconditioner $\fC_k$ is introduced. The general idea is to convert the saddle point problem into an SPD problem. Then the well-established multigrid theories for SPD system \cite{brenner2005convergence,Hackbusch,bramble1993multigrid} can be utilized. We omit some proofs since they are identical to those of \cite{brenner2020multigrid}.
\subsection{Set-Up}

Let the triangulation $\mathcal{T}_1, \mathcal{T}_2, ...$ be generated from the triangulation $\mathcal{T}_0$ through uniform subdivisions such that $h_k\approx\frac12 h_{k-1}$ and $V_k$ be the DG space associated with $\mathcal{T}_k$. Our goal is to design multigrid methods for the problem of finding $(p_k,y_k)\in V_k\times V_k$ such that
\begin{equation}\label{eq:kth}
\mathcal{B}_k((p_k,y_k),(q,z))=F(q)+G(z) \quad \quad \forall (q,z)\in V_k\times V_k,
\end{equation}
where $F, G\in V_k'$, and for the dual problem of finding $(p_k,y_k)\in V_k\times V_k$ such that
\begin{equation}\label{eq:kthdual}
\mathcal{B}_k((q,z),(p_k,y_k))=F(q)+G(z) \quad \quad \forall (q,z)\in V_k\times V_k.
\end{equation}
Here $\cB_k$ represents the bilinear form $\cB_h$ on $\FK$.

Let $(\cdot,\cdot)_k$ be a mesh-dependent inner product on $V_k$ 
\begin{equation}
(v,w)_k=h_k^n\sum_{i=1}^{n_k}v(p_i)w(p_i) \quad \forall v, w\in V_k,
\end{equation}
where $h_k=\max_{T\in \mathcal{T}_k} \mbox{diam}T$, $\{p_i\}_{i=1}^{n_k}$ are the nodes in $\mathcal{T}_k$.

\begin{remark}
It can be shown that $(v,v)_k\approx\|v\|^2_{\LT}$ for $v\in V_k$ (cf. \cite{BS}).
\end{remark}

Then the mesh-dependent inner product $[\cdot,\cdot]_k$ on $V_k\times V_k$ is defined by
\begin{equation}
[(p,y),(q,z)]_k=(p,q)_k+(y,z)_k.
\end{equation}
It is easy to see
\begin{equation}\label{eq:l2normeq}
[(p,y),(p,y)]_k\approx\|p\|^2_{\LT} +\|y\|^2_{\LT}  \quad \forall (p,y)\in V_k\times V_k.
\end{equation}

The coarse-to-fine operator $I^k_{k-1}: V_{k-1}\times V_{k-1}\longrightarrow V_k\times V_k$ is the natural injection and the fine-to-coarse operator $I^{k-1}_k:V_k\times V_k\longrightarrow V_{k-1}\times V_{k-1}$ is the transpose of $I^k_{k-1}$ with respect to the mesh-dependent inner product. Indeed,
\begin{equation*}
[I^{k-1}_k(p,y),(q,z)]_{k-1}=[(p,y),I^k_{k-1}(q,z)]_k \quad \forall (p,y)\in V_k\times V_k, (q,z)\in V_{k-1}\times V_{k-1}.
\end{equation*}

Let the system operator $\fB_k:V_k\times V_k\longrightarrow V_k\times V_k$ be defined by
\begin{equation}
[\fB_k(p,y),(q,z)]_k=\mathcal{B}_k((p,y),(q,z)) \quad \forall (p,y),(q,z) \in V_k\times V_k.
\end{equation}
Then the $k$-th level problem \eqref{eq:kth} is equivalent to
\begin{equation}\label{eq:generalproblem}
\fB_k(p,y)=(f,g),
\end{equation}
where $(f,g)\in V_k\times V_k$ is defined by
\begin{equation}
    [(f,g),(q,z)]_k=F(q)+G(z)\quad\forall (q,z)\in V_k\times V_k,
\end{equation}
and the dual problem \eqref{eq:kthdual} becomes
\begin{equation}\label{eq:generalproblemdual}
\fB_k^t(p,y)=(f,g).
\end{equation}
Here, for all $(p,y),(q,z)\in\FK$, we have 
\begin{equation}
    [\fB_k^t(p,y),(q,z)]_k=[(p,y),\fB_k(q,z)]_k=\cB_k((q,z),(p,y)).
\end{equation}

\subsection{A Block-Diagonal Preconditioner}

Let $L_k: V_k\rightarrow V_k$ be an operator that is SPD with respect to $(\cdot,\cdot)_k$ and satisfies 
\begin{equation}
(L_kv,v)_k\approx\beta^\frac12\|v\|_{1,h}^2+\|v\|_\LT^2\quad \forall v\in V_k.
\end{equation}
Then the operator $\mathfrak{C}_k:V_k\times V_k\longrightarrow V_k\times V_k$ given by
\begin{equation}
\mathfrak{C}_k(p,y)=(L_kp,L_ky)
\end{equation}
is SPD with respect to $[\cdot,\cdot]_k$ and we have
\begin{equation}
[\mathfrak{C}_k(p,y),(p,y)]_k\approx\|p\|^2_{\HBh} +\|y\|^2_{\HBh} \quad \quad \forall (p,y)\in V_k\times V_k.
\end{equation}
Here the hidden constant is independent of $k$ and $\beta$.

\begin{remark}
 In practice, we use $\fC_k^{-1}$ as a block preconditioner. The operation $L_k^{-1}\phi$ can be computed approximately by solving the following boundary value problem
\begin{equation}\label{eq:ckinv}
\begin{aligned}
    -\beta^\frac12\Delta u+u&=\phi\quad\text{in}\quad\O,\\
    u&=0\quad\text{on}\quad\partial\O
\end{aligned}
\end{equation}
using a SIP discretization. This can be constructed by multigrid \cite{brenner2005convergence,gopalakrishnan2003multilevel}.
\end{remark}

\begin{lemma}\label{eqnorm}
We have 
\begin{align}
[\fB^t_k\mathfrak{C}^{-1}_k\fB_k(p,y),(p,y)]_k&\approx\|p\|^2_{\HBh} +\|y\|^2_{\HBh} \quad \forall (p,y)\in V_k\times V_k,\label{eq:ckeq}\\
[\fB_k\mathfrak{C}^{-1}_k\fB^t_k(p,y),(p,y)]_k&\approx\|p\|^2_{\HBh} +\|y\|^2_{\HBh} \quad \forall (p,y)\in V_k\times V_k.\label{eq:ckeqdual}
\end{align}
\end{lemma}

\begin{lemma}\label{lem:lambda}
There exists positive constants $C_{\text{min}}$ and $C_{\text{max}}$, independent of $k$ and $\beta$, such that
\begin{alignat}{3}
&\lambda_{\text{min}}(\fB^t_k\mathfrak{C}_k^{-1}\fB_k)\ge C_{\text{min}},&&\quad\lambda_{\text{min}}(\fB_k\mathfrak{C}_k^{-1}\fB^t_k)\ge C_{\text{min}},\label{eq:eigenlb}\\
&\lambda_{\text{max}}(\fB^t_k\mathfrak{C}_k^{-1}\fB_k)\le C_{\text{max}}(\beta^{\frac{1}{2}}h_k^{-2}+1),&&\quad \lambda_{\text{max}}(\fB_k\mathfrak{C}_k^{-1}\fB^t_k)\le C_{\text{max}}(\beta^{\frac{1}{2}}h_k^{-2}+1).\label{eq:eigenub}
\end{alignat}
\end{lemma}

\begin{proof}
    We only prove the estimates involving $\fB^t_k\mathfrak{C}_k^{-1}\fB_k$, other estimates are similar. It follows from \eqref{eq:ckeq} and \eqref{eq:l2normeq} that
    \begin{equation}
        [\fB^t_k\mathfrak{C}_k^{-1}\fB_k(p,y),(p,y)]_k\ge \|p\|^2_\LT+\|y\|^2_\LT\approx[(p,y),(p,y)]_k. 
    \end{equation}
    The estimate \eqref{eq:eigenlb} then is trivial by Rayleigh quotient formula.
    We also have, by \eqref{eq:ckeq}, \eqref{eq:dginverse} and \eqref{eq:l2normeq}, that
    \begin{equation}
        [\fB^t_k\mathfrak{C}_k^{-1}\fB_k(p,y),(p,y)]_k\lesssim (\beta^{\frac{1}{2}}h_k^{-2}+1)[(p,y),(p,y)]_k. 
    \end{equation}
    The estimate \eqref{eq:eigenub} is immediate by Rayleigh quotient formula.
\end{proof}

\begin{remark}\label{rem:bcbwellcon}
    Lemma \ref{lem:lambda} implies that the operators $\fB^t_k\mathfrak{C}_k^{-1}\fB_k$ and $\fB_k\mathfrak{C}_k^{-1}\fB^t_k$ are well-conditioned when $\beta^{\frac{1}{2}}h_k^{-2}\le1$.
\end{remark}

\subsection{$W$-cycle algorithm}\label{sec:wcycle}

Let the output of the $W$-cycle algorithm for \eqref{eq:generalproblem} with initial guess $(p_0, y_0)$ and $m_1$ (resp. $m_2$) pre-smoothing (resp. post-smoothing) steps be denoted by $MG_W(k,(f,g), (p_0, y_0),m_1,m_2)$.

We use a direct solve for $k=0$, i.e., we take  $MG_W(0,(f,g), (p_0, y_0),m_1,m_2)$ to be $\cB_0^{-1}(f,g)$. For $k\ge1$, we compute  $MG_W(k,(f,g), (p_0, y_0),m_1,m_2)$ in three steps.
\
\\

\noindent {\em Pre-Smoothing}\quad The approximate solutions $(p_1,y_1), \ldots, (p_{m_1},y_{m_1})$ are computed recursively by

\begin{equation}\label{eq:pre}
(p_j,y_j)=(p_{j-1},y_{j-1})+\lambda_k\mathfrak{C}_k^{-1}\fB^t_k((f,g)-\fB_k(p_{j-1},y_{j-1}))
\end{equation}
for $1\le j \le m_1$. The choice of the damping factor $\lambda_k$ is determined by the following criteria,
\begin{subequations}\label{eq:lambdak}
\begin{align}
    \lambda_k&=\frac{C}{\beta^{\frac12}h_k^{-2}+1}\quad\text{when}\quad \beta^{\frac12}h_k^{-2}\ge1,\label{eq:lambdak1}\\
    \lambda_k&=\frac{2}{\lambda_{min}+\lambda_{max}}\quad\text{when}\quad \beta^{\frac12}h_k^{-2}<1,\label{eq:lambdak2}
\end{align}
\end{subequations}
where $\lambda_{min}$ and $\lambda_{max}$ are the smallest and largest eigenvalues of $\fB^t_k\mathfrak{C}_k^{-1}\fB_k$ respectively.
\
\\

\noindent {\em Coarse Grid Correction}\quad Let $(f',g')=I^{k-1}_k((f,g)-\fB_k(p_{m_1}, y_{m_1}))$ be the transferred residual of $(p_{m_1},y_{m_1})$ and compute $(p'_1,y'_1),(p'_2,y'_2)\in V_{k-1}\times V_{k-1}$ by 
\begin{eqnarray}
(p'_1,y'_1)&=&MG_W(k-1,(f',g'),(0,0),m_1,m_2),\\
(p'_2,y'_2)&=&MG_W(k-1,(f',g'),(p'_1,y'_1),m_1,m_2).
\end{eqnarray}
We then take $(p_{m_1+1},y_{m_1+1})$ to be $(p_{m_1},y_{m_1})+I^k_{k-1}(p'_2,y'_2)$.
\
\\

\noindent {\em Post-Smoothing}\quad The approximate solutions $(p_{m_1+2},y_{m_1+2}), \ldots, (p_{m_1+m_2+1},y_{m_1+m_2+1})$ are computed recursively by

\begin{equation}\label{eq:post}
(p_j,y_j)=(p_{j-1},y_{j-1})+\lambda_k\fB^t_k\mathfrak{C}_k^{-1}((f,g)-\fB_k(p_{j-1},y_{j-1}))
\end{equation}
for $m_1+2\le j \le m_1+m_2+1$.

The final output is $MG_W(k,(f,g), (p_0, y_0),m_1,m_2)=(p_{m_1+m_2+1},y_{m_1+m_2+1})$.

\begin{remark}
    The choice of \eqref{eq:lambdak1} is motivated by Lemma \ref{lem:lambda} such that \linebreak$\lambda_{max}(\lambda_k\fB^t_k\mathfrak{C}_k^{-1}\fB_k)\le 1$. The choice of \eqref{eq:lambdak2} is motivated by the optimal choice of Richardson iteration \cite{Saad} and the well-conditioning of $\fB^t_k\mathfrak{C}_k^{-1}\fB_k$ (cf. Remark \ref{rem:bcbwellcon}).
\end{remark}

\begin{remark}
    Note that the post-smoothing step is the Richardson iteration of the SPD system
    \begin{equation}
        \fB^t_k\mathfrak{C}_k^{-1}\fB_k(p,y)=\fB^t_k\mathfrak{C}_k^{-1}(f,g)
    \end{equation}
    which is equivalent to \eqref{eq:generalproblem}.
\end{remark}

\subsection{$V$-cycle algorithm}\label{sec:vcycle}
Let the output of the $V$-cycle algorithm for \eqref{eq:generalproblem} with initial guess $(p_0, y_0)$ and $m_1$ (resp. $m_2$) pre-smoothing (resp. post-smoothing) steps be denoted by $MG_V(k,(f,g), (p_0, y_0),m_1,m_2)$. 

The computation of $MG_V(k,(f,g),(p_0, y_0),m_1,m_2)$ differs from the computation of the $W$-cycle algorithm only in the coarse grid correction step, where we compute
\begin{equation}
(p'_1,y'_1)=MG_V(k-1,(f',g'),(0,0),m_1,m_2)
\end{equation}
and take $(p_{m_1+1},y_{m_1+1})$ to be $(p_{m_1},y_{m_1})+I^k_{k-1}(p'_1,y'_1)$.

\subsection{Multigrid algorithms for \eqref{eq:generalproblemdual}}
We define $W$-cycle and $V$-cycle algorithms for \eqref{eq:generalproblemdual} by simply interchanging $\fB^t_k$ and $\fB_k$ in sections \ref{sec:wcycle} and \ref{sec:vcycle}. The pre-smoothing step is given by
 \begin{equation}\label{predual}
(p_j,y_j)=(p_{j-1},y_{j-1})+\lambda_k\mathfrak{C}_k^{-1}\fB_k((f,g)-\fB^t_k(p_{j-1},y_{j-1}))
\end{equation}
and the post-smoothing step is given by
\begin{equation}\label{eq:postdual}
(p_j,y_j)=(p_{j-1},y_{j-1})+\lambda_k\fB_k\mathfrak{C}_k^{-1}((f,g)-\fB^t_k(p_{j-1},y_{j-1})).
\end{equation}
 
 \section{Smoothing and Approximation Properties}\label{sec:smapprox}
 
 In this section, we establish the smoothing property and the approximation property of the $W$-cycle algorithm. These results can then be used to establish the convergence of $W$-cycle algorithm as in \cite{brenner2020multigrid}. We omit some proofs since they are identical to those of \cite{brenner2020multigrid}.

\subsection{A Scale of Mesh-Dependent Norms}
For $0\le s\le1$, we define
\begin{alignat}{3}
\trinorm{(p,y)}_{s,k}&=[(\fB^t_k\mathfrak{C}_k^{-1}\fB_k)^s(p,y),(p,y)]^{\frac{1}{2}}_k&&\quad \forall (p,y) \in V_k\times V_k,\label{norms}\\
\trinorm{(p,y)}^{\widetilde{ }}_{s,k}&=[(\fB_k\mathfrak{C}_k^{-1}\fB^t_k)^s(p,y),(p,y)]^{\frac{1}{2}}_k&&\quad \forall (p,y) \in V_k\times V_k\label{normsdual}.
\end{alignat}
Note that
\begin{equation}\label{eq:0knormeq}
    \trinorm{(p,y)}^2_{0,k}\approx\|p\|^2_\LT+\|y\|^2_\LT\approx(\trinorm{(p,y)}^{\widetilde{ }}_{0,k})^2\quad\forall (p,y)\in\FK,
\end{equation}
by \eqref{eq:l2normeq} and
\begin{equation}\label{eq:1knormeq}
    \trinorm{(p,y)}^2_{1,k}\approx\|p\|^2_\HBh+\|y\|^2_\HBh\approx(\trinorm{(p,y)}^{\widetilde{ }}_{1,k})^2\quad\forall (p,y)\in\FK,
\end{equation}
by \eqref{eq:ckeq} and \eqref{eq:ckeqdual}.

\subsection{Post-Smoothing Properties}
The error propagation operator for one post-smoothing step defined by \eqref{eq:post} is given by
\begin{equation}\label{eq:postop}
R_k=Id_k-\lambda_k\fB^t_k\mathfrak{C}_k^{-1}\fB_k,
\end{equation}
where $Id_k$ is the identity operator on $V_k\times V_k$.
We also need the error propagation operator for one post-smoothing step defined by \eqref{eq:postdual} which is
\begin{equation}\label{eq:postopdual}
\tilde{R}_k=Id_k-\lambda_k\fB_k\mathfrak{C}_k^{-1}\fB^t_k.
\end{equation}

\begin{lemma}[Smoothing properties]\label{lem:postsm}
In the case where $\beta^\frac12 h_k^{-1}<1$, we have,
\begin{align}
    \trinorm{R_k(p,y)}_{1,k}&\le \tau\trinorm{(p,y)}_{1,k}\quad \forall (p,y)\in \FK,\\
    \trinorm{\tilde{R}_k(p,y)}^{\widetilde{ }}_{1,k}&\le \tau\trinorm{(p,y)}^{\widetilde{ }}_{1,k}\quad \forall (p,y)\in \FK,
\end{align}
where $\tau\in(0,1)$ is independent of $k$ and $\beta$.

In the case where $\beta^\frac12 h_k^{-1}\ge1$, we have, for $0\le s\le 1$,
\begin{align}
\trinorm{R_k^m(p,y)}_{1,k}&\le C(\beta^{\frac{1}{2}}h_k^{-2}+1)^{s/2}m^{-s/2}\trinorm{(p,y)}_{1-s,k} \quad \forall (p,y)\in \FK,\\
\trinorm{\tilde{R}_k^m(p,y)}^{\widetilde{ }}_{1,k}&\le C(\beta^{\frac{1}{2}}h_k^{-2}+1)^{s/2}m^{-s/2}\trinorm{(p,y)}^{\widetilde{ }}_{1-s,k} \quad \forall (p,y)\in \FK,
\end{align}
where the constant $C$ is independent of $k$ and $\beta$.
\end{lemma}

\subsection{Approximation Properties}
The operators $P^{k-1}_k:V_k\times V_k\rightarrow V_{k-1}\times V_{k-1}$ and $\tilde{P}^{k-1}_k:V_k\times V_k\rightarrow V_{k-1}\times V_{k-1}$ are defined as follows. For all $(p,y)\in V_k\times V_k, (q,z)\in V_{k-1}\times V_{k-1}$
\begin{align}
\mathcal{B}_{k-1}(P^{k-1}_k(p,y),(q,z))&=\mathcal{B}_{k}((p,y),I_{k-1}^k(q,z))=\mathcal{B}_k((p,y),(q,z)),\label{eq:project}\\
\mathcal{B}_{k-1}((q,z),\tilde{P}^{k-1}_k(p,y))&=\mathcal{B}_{k}(I_{k-1}^k(q,z),(p,y))=\mathcal{B}_k((q,z),(p,y))\label{eq:projectdual}.
\end{align}

\begin{lemma}
We have the following properties,
\begin{alignat}{3}
&(I^k_{k-1}P^{k-1}_k)^2=I^k_{k-1}P^{k-1}_k,\quad&&(Id_k-I^k_{k-1}P^{k-1}_k)^2=Id_k-I^k_{k-1}P^{k-1}_k,\label{eq:IP1}\\
&(I^k_{k-1}\tilde{P}^{k-1}_k)^2=I^k_{k-1}\tilde{P}^{k-1}_k,\quad&&(Id_k-I^k_{k-1}\tilde{P}^{k-1}_k)^2=Id_k-I^k_{k-1}\tilde{P}^{k-1}_k.\label{eq:IP2}
\end{alignat}
\end{lemma}

\begin{proof}
    We only prove the identities \eqref{eq:IP1}, the ones in \eqref{eq:IP2} are similar. Notice that for $(p,y)\in\FKm$, we have
    \begin{equation}
        \cB_{k-1}(P^{k-1}_kI^k_{k-1}(p,y),(q,z))=\cB_k(I^k_{k-1}(p,y),I^k_{k-1}(q,z))=\cB_{k-1}((p,y),(q,z))
    \end{equation}
    for all $(q,z)\in\FKm$. This implies $P^{k-1}_kI^k_{k-1}=Id_{k-1}$. Then the equalities \eqref{eq:IP1} are immediate by a direct calculation.
\end{proof}

\begin{lemma}[Approximation Properties]\label{lem:approx}
We have, for all $(p,y)\in V_k \times V_k$ and $k\ge1$,
\begin{align}
\trinorm{(Id_k-I^k_{k-1}P^{k-1}_k)(p,y)}_{0,k}&\lesssim ({\beta}^{\frac{1}{2}}h_k^{-2}+1)^{\frac{1}{2}}\beta^{-\frac{1}{2}}h_k^2\trinorm{(p,y)}_{1,k}, \label{eq:approx}\\
\trinorm{(Id_k-I^k_{k-1}\tilde{P}^{k-1}_k)(p,y)}^{\widetilde{ }}_{0,k}&\lesssim ({\beta}^{\frac{1}{2}}h_k^{-2}+1)^{\frac{1}{2}}\beta^{-\frac{1}{2}}h_k^2\trinorm{(p,y)}^{\widetilde{ }}_{1,k}.\label{eq:approxdual}
\end{align}
Here the hidden constant is independent of $k$ and $\beta$.
\end{lemma}

\begin{proof}
We will only prove \eqref{eq:approx}, the argument for \eqref{eq:approxdual} is similar.
Let $(p,y)\in V_k \times V_k$ be arbitrary and $(\zeta,\mu)=(Id_k-I^k_{k-1}P^{k-1}_k)(p,y)$. By (\ref{eq:l2normeq}), it suffices to show that
\begin{equation}
\|\zeta\|_{\LT}+\|\mu\|_{\LT}\lesssim ({\beta}^{\frac{1}{2}}h_k^{-2}+1)^{\frac{1}{2}}\beta^{-\frac{1}{2}}h_k^2\trinorm{(p,y)}_{1,k}.
\end{equation}
The estimate \eqref{eq:approx} is established through a duality argument (cf. \cite{brenner2005convergence}). Let $(\xi,\theta)\in H^1_0(\Omega)\times H^1_0(\Omega)$ satisfies
\begin{equation}\label{eq:approxduality}
\mathcal{B}((q,z),(\xi,\theta))=(\zeta,q)_\LT+(\mu,z)_\LT \quad \forall (q,z)\in \ES.
\end{equation}

Moreover, we define $(\xi_{k},\theta_{k})\in \FK$ and $(\xi_{k-1},\theta_{k-1})\in V_{k-1}\times V_{k-1}$ by
\begin{align}
\mathcal{B}_{k}((q,z),(\xi_{k},\theta_{k})&=(\zeta,q)_\LT+(\mu,z)_\LT \quad \forall (q,z)\in V_{k}\times V_{k},\label{eq:approxdvk}\\
\mathcal{B}_{k-1}((q,z),(\xi_{k-1},\theta_{k-1}))&=(\zeta,q)_\LT+(\mu,z)_\LT \quad \forall (q,z)\in V_{k-1}\times V_{k-1}.\label{eq:approxdvkm1}
\end{align}
Note that \eqref{eq:approxdvk} and \eqref{eq:approxdvkm1} implies
\begin{equation}\label{eq:approxbbrela}
    \mathcal{B}_{k}((q,z),(\xi_{k},\theta_{k}))=\mathcal{B}_{k-1}((q,z),(\xi_{k-1},\theta_{k-1}))\quad\forall (q,z)\in V_{k-1}\times V_{k-1}.
\end{equation}
It follows from \eqref{eq:approxbbrela} and \eqref{eq:projectdual} that
\begin{equation}\label{eq:approxkkm1}
    (\xi_{k-1},\theta_{k-1})=\tilde{P}_k^{k-1}(\xi_{k},\theta_{k}).
\end{equation}

We have the following by \eqref{eq:hb1esti} and $h_k\approx\frac12h_{k-1}$,
\begin{equation}\label{eq:approxh1esti1}
    \begin{aligned}
        \|\xi-\xi_{k}\|_{\HBh}&+\|\theta-\theta_{k}\|_{\HBh}\\
        &\le C_{\Omega}({\beta}^{\frac{1}{2}}h_k^{-2}+1)^{\frac{1}{2}}\beta^{-\frac{1}{2}}h_k^2(\|\zeta\|_{\LT}+\|\mu\|_{\LT}),
    \end{aligned}
\end{equation}
\begin{equation}\label{eq:approxh1esti2}
    \begin{aligned}
        \|\xi-\xi_{k-1}\|_{\HBh}&+\|\theta-\theta_{k-1}\|_{\HBh}\\
        &\le C_{\Omega}({\beta}^{\frac{1}{2}}h_k^{-2}+1)^{\frac{1}{2}}\beta^{-\frac{1}{2}}h_k^2(\|\zeta\|_{\LT}+\|\mu\|_{\LT}).
    \end{aligned}
\end{equation}
Therefore by \eqref{eq:bhwposed}, \eqref{eq:project}, \eqref{eq:approxdvk}, \eqref{eq:approxkkm1}, \eqref{eq:approxh1esti1} and \eqref{eq:approxh1esti2} we have
\begin{align*}
\|\zeta\|^2_{\LT}+\|\mu\|^2_{\LT} =& \mathcal{B}_k((\zeta,\mu),(\xi_k,\theta_k))\\
=& \mathcal{B}_k((Id_k-I^k_{k-1}P^{k-1}_k)(p,y),(\xi_k,\theta_k))\\
=& \mathcal{B}_k((p,y),(\xi_k,\theta_k))-\cB_k(I^k_{k-1}P^{k-1}_k(p,y),(\xi_k,\theta_k))\\
=& \mathcal{B}_k((p,y),(\xi_k,\theta_k))-\cB_{k-1}(P^{k-1}_k(p,y), \tilde{P}^{k-1}_k(\xi_k,\theta_k))\\
=& \mathcal{B}_k((p,y),(\xi_k,\theta_k))-\cB_{k-1}(P^{k-1}_k(p,y), (\xi_{k-1},\theta_{k-1}))\\
=& \mathcal{B}_k((p,y),(\xi_k,\theta_k))-\cB_{k}((p,y),I^k_{k-1}(\xi_{k-1},\theta_{k-1}))\\
=& \mathcal{B}_k((p,y),(\xi_k,\theta_k)-I^k_{k-1}(\xi_{k-1},\theta_{k-1}))\\
\lesssim&(\|\xi_k-\xi_{k-1}\|^2_{\HBh}+\|\theta_k-\theta_{k-1}\|^2_{\HBh})^{\frac{1}{2}}\\
&\times(\|p\|^2_{\HBh}+\|y\|^2_{\HBh})^{\frac{1}{2}}\\
\lesssim& ({\beta}^{\frac{1}{2}}h_k^{-2}+1)^{\frac{1}{2}}\beta^{-\frac{1}{2}}h_k^2(\|\zeta\|_{\LT}+\|\mu\|_{\LT})\trinorm{(p,y)}_{1,k}.
\end{align*}
\end{proof}

\section{Convergence Analysis}\label{sec:convanalysis}

Let $E_k:\FK\rightarrow\FK$ be the error propagation operator for the $k$th level $W$-cycle algorithm for \eqref{eq:generalproblem}. The following recursive relationship is well-known (cf. \cite{Hackbusch,BS}),
\begin{equation}\label{eq:ek}
    E_k=R_k^{m_2}(Id_k-I^k_{k-1}P^{k-1}_k+I_{k-1}^kE_{k-1}^2P_k^{k-1})S_k^{m_1},
\end{equation}
where $R_k$ is defined in \eqref{eq:postop} and 
\begin{equation}
S_k=Id_k-\lambda_k\mathfrak{C}^{-1}_k\fB_k^t\fB_k
\end{equation}
measures the effect of one pre-smoothing step \eqref{eq:pre}.
\begin{remark}
We have the following adjoint relation
\begin{equation}\label{eq:skrkduality}
\mathcal{B}_k(S_k(p,y),(q,z))=\mathcal{B}_k((p,y),\tilde{R}_k(q,z)) \quad \forall (p,y),(q,z)\in \FK,
\end{equation}
where $\tilde{R}_k$ is defined in \eqref{eq:postopdual}. The relation \eqref{eq:skrkduality} is the reason why we consider the multigrid algorithms for \eqref{eq:generalproblem} and \eqref{eq:generalproblemdual} simultaneously.
\end{remark}

\begin{lemma}\label{lem:skeq}
    We have
    \begin{equation}
        \|(Id_k-I^k_{k-1}P^{k-1}_k)S^m_k\|\approx\|\tilde{R}^m_k(Id_k-I^k_{k-1}\tilde{P}^{k-1}_k)\|
    \end{equation}
    where $\|\cdot\|$ is the operator norm with respect to $\trinorm{\cdot}_{1,k}$.
\end{lemma}

\begin{proof}
For all $(p,y)\in \FK$, it follows from $\eqref{eq:bhwposed}$ that
\begin{equation}
    \begin{aligned}
            \trinorm{(Id_k-&I^k_{k-1}P^{k-1}_k)S^m_k(p,y)}_{1,k}\\
            &\approx\sup_{(q,z)\in \FK} \frac{\mathcal{B}_k((Id_k-I^k_{k-1}P^{k-1}_k)S^m_k(p,y),(q,z))}{\trinorm{(q,z)}_{1,k}}\\
            &=\sup_{(q,z)\in \FK} \frac{\mathcal{B}_k((p,y),\tilde{R}_k^m(Id_k-I^k_{k-1}\tilde{P}^{k-1}_k)(q,z))}{\trinorm{(q,z)}_{1,k}}\\
            &\lesssim\trinorm{(p,y)}_{1,k}\|R^m_k(Id_k-I^k_{k-1}\tilde{P}^{k-1}_k)\|.
    \end{aligned}
\end{equation}
This implies $\|(Id_k-I^k_{k-1}P^{k-1}_k)S^m_k\|\lesssim\|\tilde{R}^m_k(Id_k-I^k_{k-1}\tilde{P}^{k-1}_k)\|$. The other direction of the estimate is similar.
\end{proof}

\subsection{Convergence of the two-grid algorithm.}

In the two-grid algorithm the coarse grid residual equation is solved exactly ($E_{k-1}=0$ in \eqref{eq:ek}). We therefore obtain the error propagation operator of the two-grid algorithm $R_k^{m_2}(Id_k-I^k_{k-1}P^{k-1}_k)S_k^{m_1}$ with $m_1$ pre-smoothing steps and $m_2$ post-smoothing steps.

We have the following lemma for the convergence of the two-grid algorithm.

\begin{lemma}\label{lem:twogrid}
    In the case of $\beta^{-\frac{1}{2}}h_k^2<1$, we have
    \begin{equation}\label{eq:twogridconv1}
        \|R_k^{m_2}(Id_k-I^k_{k-1}P^{k-1}_k)S_k^{m_1}\|\lesssim \tau^{m_1+m_2}.
    \end{equation}
    In the case of $\beta^{-\frac{1}{2}}h_k^2\ge 1$, we have
    \begin{equation}\label{eq:twogridconv2}
        \|R_k^{m_2}(Id_k-I^k_{k-1}P^{k-1}_k)S_k^{m_1}\|\lesssim [\max(1,{m_2})\max(1,{m_1})]^{-1/2}.
    \end{equation}
\end{lemma}

\begin{proof}
    The proof can be found in \cite[Section 5.1]{brenner2020multigrid}. The key ingredients are \eqref{eq:IP2}, Lemma \ref{lem:postsm}, Lemma \ref{lem:approx} and Lemma \ref{lem:skeq}.
\end{proof}

\subsection{Convergence of the $W$-Cycle Algorithm}

It is well-known that the convergence of the two-grid algorithm implies the convergence of the $W$-cycle algorithm by a standard perturbation argument (cf. \cite{Hackbusch,briggs2000multigrid,BS}). A delicate modification \cite{brenner2020multigrid} of the standard argument leads to the following theorem.

\begin{theorem}\label{thm:wcycle}
    There exists a positive integer $m_\ast$ independent of $k$ and $\beta$ such that
    \begin{alignat}{3}
        \|E_k\|&\le C_\sharp\tau^{m_1+m_2} &&\quad\forall 1\le k\le k_\ast,\label{eq:wcyclecoar}\\ 
        \|E_k\|&\le C_\flat[\max(1,{m_2})\max(1,{m_1})]^{-1/2}+4^{1-2^{k-k_\ast}}(C_\sharp\tau^{m_1+m_2})&&\quad\forall k\ge k_\ast+1.\label{eq:wcyclefine}
    \end{alignat}
 provided $[\max(1,{m_2})\max(1,{m_1})]\ge m_\ast$. Here $C_\sharp$ and $C_\flat$ are constants independent of $k$ and $\beta$ and the integer $k_\ast$ is the largest positive integer such that $\beta^\frac12h_k^{-2}<1$. 
\end{theorem}

\begin{remark}\label{rem:wcycle}
The interpretations and implications of Theorem \ref{thm:wcycle} are as follows.
\begin{itemize}
    \item[1.] The $W$-cycle algorithm for the $k$th level problem \eqref{eq:generalproblem} is a contraction in the energy norm $\trinorm{\cdot}_{1,k}$ if the number of smoothing steps is large enough. The contraction number is bounded away from $1$ uniformly in $k$ and $\beta$. Therefore, the $W$-cycle algorithm is robust with respect to $k$ and $\beta$.
    \item[2.] At coarse levels (where $\beta^\frac12h_k^{-2}<1$), the estimate $\eqref{eq:wcyclecoar}$ indicates that the contraction numbers decrease exponentially with respect to the number of smoothing steps. The estimate $\eqref{eq:wcyclefine}$ implies that the contraction numbers will be dominated by the term $[\max(1,{m_2})\max(1,{m_1})]^{-1/2}$ at finer levels (where $\beta^\frac12h_k^{-2}\ge1$) eventually. 
\end{itemize} 
\end{remark}

\section{Numerical Results}\label{sec:numerics}

In this section, we report the numerical results of the symmetric $W$-cycle and $V$-cycle algorithms $(m_1=m_2=m)$. The preconditioner $\fC_k^{-1}$ is computed using a $V(4,4)$ multigrid solve for \eqref{eq:ckinv} based on a SIP discretization \cite{brenner2005convergence}. The eigenvalues $\lambda_{max}$ and $\lambda_{min}$ in \eqref{eq:lambdak2} are estimated using power iterations. We employed the MATLAB\textbackslash C$++$ toolbox FELICITY \cite{walker2018felicity} in our computation.

\begin{example}[Unit Square]\label{ex:unitsq}
    In this example we take $\O=(0,1)^2$ and $\sigma=6$ in \eqref{eq:dgbilinearsip}. For simplicity, we take $\bm{\zeta}=[1,0]^t$ and $\gamma=0$ in \eqref{eq:ardef}. See Figure \ref{fig:usquaretri} for the initial triangulation $\cT_0$ and the uniform refinements $\cT_1$ and $\cT_2$.
\end{example}

We report the contraction numbers of the $W$-cycle algorithm in Tables \ref{table:WCycle11}-\ref{table:WCycle13} for $\beta=10^{-2}$, $\beta=10^{-4}$ and $\beta=10^{-6}$. We observe that the contraction numbers of the symmetric $W$-cycle algorithm decay exponentially at coarse levels and then approach the standard $O(m^{-1})$ behavior at finer levels for all choices of $\beta$. Notice that our $W$-cycle algorithm is clearly robust with respect to $\beta$ and the performance agrees with Remark \ref{rem:wcycle}.

\begin{figure}[t]
    \centering
    \includegraphics[height=1.2in]{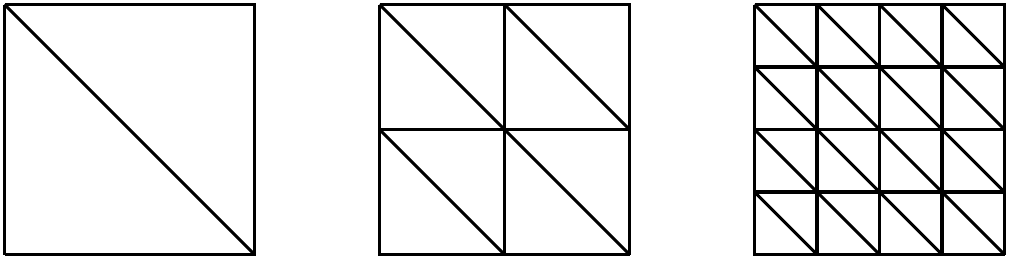}
    \caption{Triangulations $\cT_0$, $\cT_1$ and $\cT_2$ for the unit square in Example \ref{ex:unitsq}}\label{fig:usquaretri}
\end{figure}

\begin{table}[h]
\begin{tabular}{|c|c|c|c|c|c|c|c|}\hline
\backslashbox{$k$}{\lower 2pt\hbox{$m$}}&$2^0$&$2^1$&$2^2$&$2^3$&$2^4$&$2^5$&$2^6$\\
\hline
\rule{0pt}{2.5ex}$1$&8.17e-01&6.87e-01&5.08e-01&2.94e-01&1.02e-01&1.25e-02&1.89e-04\\
\hline
\rule{0pt}{2.5ex}$2$&8.31e-01&7.02e-01&5.31e-01&3.29e-01&1.43e-01&4.19e-02&8.24e-03\\
\hline
\rule{0pt}{2.5ex}$3$&8.96e-01&8.08e-01&6.74e-01&4.84e-01&2.92e-01&1.31e-01&4.48e-02\\
\hline
\rule{0pt}{2.5ex}$4$&8.64e-01&7.55e-01&5.93e-01&4.05e-01&2.17e-01&1.01e-01&4.58e-02\\
\hline
\rule{0pt}{2.5ex}$5$&8.49e-01&7.36e-01&5.63e-01&3.71e-01&1.95e-01&9.95e-02&4.59e-02\\
\hline
\rule{0pt}{2.5ex}$6$&8.46e-01&7.35e-01&5.55e-01&3.61e-01&1.90e-01&9.50e-02&4.68e-02\\
\hline
\rule{0pt}{2.5ex}$7$&8.45e-01&7.34e-01&5.52e-01&3.57e-01&1.90e-01&9.54e-02&4.70e-02\\
\hline
\end{tabular}
\par\bigskip
 \caption{The contraction numbers of the
 $k$-th level ($k=1,\ldots,7$) symmetric $W$-cycle algorithm for Example \ref{ex:unitsq}
 with $\beta=10^{-2}$ and $m=2^0,\ldots, 2^6$.}
\label{table:WCycle11}
\end{table}

\begin{table}[h]
\begin{tabular}{|c|c|c|c|c|c|c|c|}\hline
\backslashbox{$k$}{\lower 2pt\hbox{$m$}}&$2^0$&$2^1$&$2^2$&$2^3$&$2^4$&$2^5$&$2^6$\\
\hline
\rule{0pt}{2.5ex}$1$&5.85e-01&3.74e-01&1.48e-01&2.42e-02&6.54e-04&4.81e-07&4.97e-14\\
\hline
\rule{0pt}{2.5ex}$2$&8.08e-01&6.71e-01&4.65e-01&2.28e-01&5.52e-02&3.34e-03&1.13e-05\\
\hline
\rule{0pt}{2.5ex}$3$&8.38e-01&7.26e-01&5.36e-01&3.31e-01&1.54e-01&3.87e-02&3.66e-03\\
\hline
\rule{0pt}{2.5ex}$4$&9.36e-01&8.78e-01&7.82e-01&6.32e-01&4.38e-01&2.49e-01&1.01e-01\\
\hline
\rule{0pt}{2.5ex}$5$&8.85e-01&7.92e-01&6.54e-01&4.85e-01&2.94e-01&1.47e-01&6.19e-02\\
\hline
\rule{0pt}{2.5ex}$6$&8.95e-01&7.46e-01&5.77e-01&3.87e-01&2.09e-01&1.06e-01&5.49e-02\\
\hline
\rule{0pt}{2.5ex}$7$&8.48e-01&7.37e-01&5.58e-01&3.63e-01&1.95e-01&9.75e-02&4.89e-02\\
\hline
\end{tabular}
\par\bigskip
 \caption{The contraction numbers of the
 $k$-th level ($k=1,\ldots,7$) symmetric $W$-cycle algorithm for Example \ref{ex:unitsq}
 with $\beta=10^{-4}$ and $m=2^0,\ldots, 2^6$.}
\label{table:WCycle12}
\end{table}

\begin{table}[h]
\begin{tabular}{|c|c|c|c|c|c|c|c|}\hline
\backslashbox{$k$}{\lower 2pt\hbox{$m$}}&$2^0$&$2^1$&$2^2$&$2^3$&$2^4$&$2^5$&$2^6$\\
\hline
\rule{0pt}{2.5ex}$1$&4.18e-01&1.75e-01&3.02e-02&9.45e-04&8.71e-07&1.07e-13&1.86e-16\\
\hline
\rule{0pt}{2.5ex}$2$&4.35e-01&1.90e-01&3.60e-02&1.24e-03&1.53e-06&4.60e-13&2.30e-16\\
\hline
\rule{0pt}{2.5ex}$3$&7.06e-01&5.10e-01&2.84e-01&8.65e-02&7.88e-03&7.14e-05&5.31e-07\\
\hline
\rule{0pt}{2.5ex}$4$&8.33e-01&7.10e-01&5.22e-01&2.78e-01&9.24e-02&1.13e-02&1.76e-04\\
\hline
\rule{0pt}{2.5ex}$5$&8.43e-01&7.30e-01&5.44e-01&3.50e-01&1.79e-01&5.68e-02&1.13e-02\\
\hline
\rule{0pt}{2.5ex}$6$&9.14e-01&8.41e-01&7.24e-01&5.45e-01&3.50e-01&1.78e-01&8.27e-02\\
\hline
\rule{0pt}{2.5ex}$7$&8.70e-01&7.67e-01&6.14e-01&4.38e-01&2.52e-01&1.24e-01&5.87e-02\\
\hline
\end{tabular}
\par\bigskip
 \caption{The contraction numbers of the
 $k$-th level ($k=1,\ldots,7$) symmetric $W$-cycle algorithm for Example \ref{ex:unitsq}
 with $\beta=10^{-6}$ and $m=2^0,\ldots, 2^6$.}
\label{table:WCycle13}
\end{table}

We have also tested the symmetric $V$-cycle algorithm for $k$th level problem \eqref{eq:generalproblem} and briefly report the results in Table \ref{table:VCycle1}. We observe that our $V$-cycle algorithm is also a contraction with slightly more numbers of smoothing steps ($m=4$) and the contraction numbers are robust with respect to $k$ and $\beta$.

\begin{table}[t]
\begin{tabular}{|c|c|c|c|c|c|c|c|}\hline
\backslashbox{$m$}{\lower 2pt\hbox{$k$}}&1&2&3&4&5&6&7\\
\hline
\multicolumn{8}{|c|}{\rule{0pt}{2.5ex}$\beta=10^{-2}$}\\[2pt]
\hline
\rule{0pt}{2.5ex}$2^2$&5.08e-01&5.55e-01&6.79e-01&6.54e-01&6.50e-01&6.23e-01&6.15e-01\\
\hline
\rule{0pt}{2.5ex}$2^3$&2.94e-01&3.45e-01&4.97e-01&4.68e-01&4.83e-01&4.54e-01&4.41e-01\\
\hline
\rule{0pt}{2.5ex}$2^4$&1.02e-01&1.52e-01&2.92e-01&2.87e-01&2.88e-01&2.80e-01&2.52e-01\\
\hline
\multicolumn{8}{|c|}{\rule{0pt}{2.5ex}$\beta=10^{-4}$}\\[2pt]
\hline
\rule{0pt}{2.5ex}$2^2$&1.49e-01&4.65e-01&5.40e-01&7.84e-01&7.31e-01&7.07e-01&7.05e-01\\
\hline
\rule{0pt}{2.5ex}$2^3$&2.42e-02&2.28e-01&3.30e-01&6.33e-01&5.82e-01&5.74e-01&5.66e-01\\
\hline
\rule{0pt}{2.5ex}$2^4$&6.53e-04&5.60e-02&1.53e-01&4.38e-01&3.82e-01&3.75e-01&3.77e-01\\
\hline
\multicolumn{8}{|c|}{\rule{0pt}{2.5ex}$\beta=10^{-6}$}\\[2pt]
\hline
\rule{0pt}{2.5ex}$2^2$&3.08e-02&3.60e-02&2.72e-01&5.18e-01&5.54e-01&7.23e-01&6.94e-01\\
\hline
\rule{0pt}{2.5ex}$2^3$&9.47e-04&1.26e-03&8.61e-02&2.83e-01&3.55e-01&5.47e-01&5.31e-01\\
\hline
\rule{0pt}{2.5ex}$2^4$&7.11e-07&1.70e-06&8.02e-03&9.24e-02&1.80e-01&3.53e-01&3.33e-01\\
\hline
\end{tabular}
\par\bigskip
 \caption{The contraction numbers of the
 $k$-th level ($k=1,\ldots,7$) symmetric $V$-cycle algorithm for Example \ref{ex:unitsq}
 with  $\beta=10^{-2},10^{-4},10^{-6}$ and $m=2^2,2^3,2^4$.}
\label{table:VCycle1}
\end{table}

\begin{example}[L-shaped Domain]\label{ex:lshaped}
    We also test our multigrid methods on nonconvex domains. In this example, we take $\O=(0,1)^2\setminus(0.5,1)^2$ and $\sigma=6$ in \eqref{eq:dgbilinearsip}. We also take $\bm{\zeta}=[1,0]^t$ and $\gamma=0$ in \eqref{eq:ardef}. See Figure \ref{fig:lshaped} for the initial triangulation $\cT_0$ and the uniform refinement $\cT_1$.
\end{example}

We report the contraction numbers of the $W$-cycle algorithm in Tables \ref{table:WCyclel1}-\ref{table:WCyclel3} for $\beta=10^{-2}$, $\beta=10^{-4}$ and $\beta=10^{-6}$. We observe that the contraction numbers of the symmetric $W$-cycle algorithm decay exponentially at coarse levels and then approach the standard $O(m^{-\frac23})$ behavior for L-shaped domains at finer levels for all choices of $\beta$. Notice that our $W$-cycle algorithm is clearly robust with respect to $\beta$ and the performance agrees with Remark \ref{rem:wcycle}.

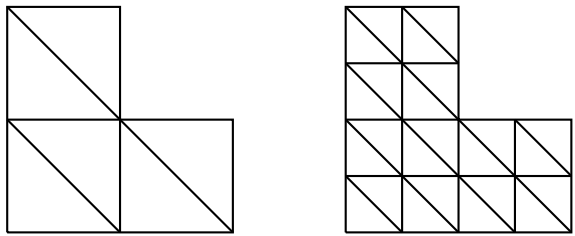
\begin{figure}
\centering
\begin{tikzpicture}
\draw[thick] plot coordinates {(0,0) (0,3) (1.5,3) (1.5,1.5) (3,1.5) (3,0) (0,0)};
\draw[thick] plot coordinates {(0,3)(3,0)};
\draw[thick] plot coordinates {(0,1.5)(1.5,1.5)};
\draw[thick] plot coordinates {(1.5,1.5) (1.5,0)};
\draw[thick] plot coordinates {(0,1.5) (1.5, 0)};

\draw[thick] plot coordinates {(4.5,0) (4.5,3) (6,3) (6,1.5) (7.5,1.5) (7.5,0) (4.5,0)};
\draw[thick] plot coordinates {(4.5,3)(7.5,0)};
\draw[thick] plot coordinates {(4.5,1.5)(6,1.5)};
\draw[thick] plot coordinates {(6,1.5) (6,0)};
\draw[thick] plot coordinates {(4.5,1.5) (6, 0)};

\draw[thick] plot coordinates {(4.5,0.75) (7.5, 0.75)};
\draw[thick] plot coordinates {(4.5,2.25) (6, 2.25)};
\draw[thick] plot coordinates {(4.5,2.25) (6.75, 0)};
\draw[thick] plot coordinates {(4.5,0.75) (5.25, 0)};
\draw[thick] plot coordinates {(5.25,3) (6, 2.25)};
\draw[thick] plot coordinates {(6.75,1.5) (7.5, 0.75)};

\draw[thick] plot coordinates {(5.25,3) (5.25, 0)};
\draw[thick] plot coordinates {(6.75,1.5) (6.75, 0)};
\end{tikzpicture}
\caption{Triangulations $\cT_0$ and $\cT_1$ for the L-shaped domain in Example \ref{ex:lshaped}} \label{fig:lshaped}
\end{figure}

\begin{table}[h]
\begin{tabular}{|c|c|c|c|c|c|c|c|}\hline
\backslashbox{$k$}{\lower 2pt\hbox{$m$}}&$2^0$&$2^1$&$2^2$&$2^3$&$2^4$&$2^5$&$2^6$\\
\hline
\rule{0pt}{2.5ex}$1$&8.34e-01&7.18e-01&5.31e-01&3.25e-01&1.30e-01&2.83e-02&1.62e-03\\
\hline
\rule{0pt}{2.5ex}$2$&8.95e-01&8.09e-01&6.74e-01&4.84e-01&2.84e-01&1.22e-01&3.10e-02\\
\hline
\rule{0pt}{2.5ex}$3$&8.64e-01&7.55e-01&5.96e-01&4.02e-01&2.17e-01&9.76e-02&3.76e-02\\
\hline
\rule{0pt}{2.5ex}$4$&8.50e-01&7.35e-01&5.65e-01&3.72e-01&1.95e-01&9.48e-02&4.58e-02\\
\hline
\rule{0pt}{2.5ex}$5$&8.46e-01&7.33e-01&5.56e-01&3.61e-01&1.91e-01&9.53e-02&4.70e-02\\
\hline
\rule{0pt}{2.5ex}$6$&8.46e-01&7.34e-01&5.53e-01&3.59e-01&1.90e-01&9.52e-02&4.62e-02\\
\hline
\rule{0pt}{2.5ex}$7$&8.45e-01&7.34e-01&5.52e-01&3.55e-01&1.91e-01&9.50e-02&4.71e-02\\
\hline
\end{tabular}
\par\bigskip
 \caption{The contraction numbers of the
 $k$-th level ($k=1,\ldots,7$) symmetric $W$-cycle algorithm for Example \ref{ex:lshaped}
 with $\beta=10^{-2}$ and $m=2^0,\ldots, 2^6$.}
\label{table:WCyclel1}
\end{table}

\begin{table}[h]
\begin{tabular}{|c|c|c|c|c|c|c|c|}\hline
\backslashbox{$k$}{\lower 2pt\hbox{$m$}}&$2^0$&$2^1$&$2^2$&$2^3$&$2^4$&$2^5$&$2^6$\\
\hline
\rule{0pt}{2.5ex}$1$&8.02e-01&6.63e-01&4.67e-01&2.26e-01&5.76e-02&3.57e-03&1.55e-05\\
\hline
\rule{0pt}{2.5ex}$2$&8.37e-01&7.20e-01&5.33e-01&3.26e-01&1.52e-01&3.67e-02&2.71e-03\\
\hline
\rule{0pt}{2.5ex}$3$&9.36e-01&8.78e-01&7.82e-01&6.29e-01&4.38e-01&2.50e-01&9.60e-02\\
\hline
\rule{0pt}{2.5ex}$4$&8.85e-01&7.90e-01&6.52e-01&4.96e-01&2.91e-01&1.44e-01&6.13e-02\\
\hline
\rule{0pt}{2.5ex}$5$&8.56e-01&7.47e-01&5.77e-01&3.87e-01&2.10e-01&1.06e-01&5.59e-02\\
\hline
\rule{0pt}{2.5ex}$6$&8.48e-01&7.37e-01&5.58e-01&3.64e-01&1.95e-01&9.73e-02&4.83e-02\\
\hline
\rule{0pt}{2.5ex}$7$&8.46e-01&7.35e-01&5.53e-01&3.56e-01&1.92e-01&9.57e-02&4.73e-02\\
\hline
\end{tabular}
\par\bigskip
 \caption{The contraction numbers of the
 $k$-th level ($k=1,\ldots,7$) symmetric $W$-cycle algorithm for Example \ref{ex:lshaped}
 with $\beta=10^{-4}$ and $m=2^0,\ldots, 2^6$.}
\label{table:WCyclel2}
\end{table}

\begin{table}[h]
\begin{tabular}{|c|c|c|c|c|c|c|c|}\hline
\backslashbox{$k$}{\lower 2pt\hbox{$m$}}&$2^0$&$2^1$&$2^2$&$2^3$&$2^4$&$2^5$&$2^6$\\
\hline
\rule{0pt}{2.5ex}$1$&4.26e-01&1.88e-01&3.61e-02&1.23e-03&1.18e-06&2.27e-13&2.26e-16\\
\hline
\rule{0pt}{2.5ex}$2$&7.07e-01&5.24e-01&2.84e-01&8.75e-02&8.40e-03&8.03e-05&3.31e-07\\
\hline
\rule{0pt}{2.5ex}$3$&8.30e-01&7.04e-01&5.15e-01&2.81e-01&9.01e-02&1.06e-02&1.53e-04\\
\hline
\rule{0pt}{2.5ex}$4$&8.42e-01&7.33e-01&5.43e-01&3.49e-01&1.78e-01&5.61e-02&1.28e-02\\
\hline
\rule{0pt}{2.5ex}$5$&9.14e-01&8.41e-01&7.23e-01&5.45e-01&3.52e-01&1.78e-01&8.26e-02\\
\hline
\rule{0pt}{2.5ex}$6$&8.70e-01&7.67e-01&6.14e-01&4.37e-01&2.50e-01&1.24e-01&5.85e-02\\
\hline
\rule{0pt}{2.5ex}$7$&8.52e-01&7.42e-01&5.66e-01&3.73e-01&2.03e-01&1.03e-01&5.28e-02\\
\hline

\end{tabular}
\par\bigskip
 \caption{The contraction numbers of the
 $k$-th level ($k=1,\ldots,7$) symmetric $W$-cycle algorithm for Example \ref{ex:lshaped}
 with $\beta=10^{-6}$ and $m=2^0,\ldots, 2^6$.}
\label{table:WCyclel3}
\end{table}

\section{Concluding Remark}\label{sec:conclude}

We proposed and analyzed multigrid methods for an elliptic optimal control problem based on DG discretizations. We proved that for a sufficiently large number of smoothing steps the $W$-cycle algorithm is uniformly convergent with respect to mesh refinements and a regularizing parameter. The numerical results coincide with the theoretical findings.

A more interesting problem is to consider an advection-dominated state equation. DG methods are promising for advection-dominated problem due to the natural built-in upwind stabilization and the weak treatment of the boundary conditions. Related work can be found, for example, in \cite{leykekhman2012local}. However, the challenge for extending our result is to design proper preconditioner so that the multigrid methods are robust for the advection-dominated case. This is under investigation in our ongoing projects.

\appendix  

\section{Proofs of \eqref{eq:ahcont} and \eqref{eq:ahcoer}}\label{apdix:pfrh}
For $T\in\mathcal{T}_h$ and $v\in H^{1+s}(\O)$ where $s\in(\frac12,1]$, the following trace inequalities with scaling is standard (cf. \cite[Lemma 7.2]{ern2017finite} and \cite[Proposition 3.1]{ciarlet2013analysis}),
    \begin{align}
         \|v\|_{L_2(\partial T)}&\le C(h_T^{-\frac12}\|v\|_{L_2(T)}+h_T^{s-\frac12}|v|_{H^s(T)}).\label{eq:traceinq}
    \end{align}
The following discrete Poinca{\'r}e inequality for DG functions \cite{brenner2003poincare,ayuso2009discontinuous,chen2004pointwise} is valid for all $v\in V_h$,
\begin{equation}\label{eq:dgpoin}
    \|v\|^2_\LT\le C\left(\sum_{T\in\O}\|\nabla v\|^2_{L_2(T)}+\sum_{e\in\partial T}\frac{1}{h_e}\|[v]\|^2_{L_2(e)}\right).
\end{equation}
\begin{proof}
    It is well-known that \cite{arnold2002unified,riviere2008discontinuous,BS}
        \begin{alignat}{3}
            \asiph(w,v)&\le C\|w\|_{1,h}\|v\|_{1,h}\quad&&\forall w,v\in V+V_h,\\
            \asiph(v,v)&\ge C\|v\|^2_{1,h}\quad&&\forall v\in V_h.\label{eq:asipcoer}
        \end{alignat}
    For the advection-reaction term, we have, for all $w,v\in V+V_h$,
    \begin{equation*}\label{eq:aarhesti}
        \begin{aligned}
            \aarh(w,v)&=\sum_{T\in\mathcal{T}_h}(\bz\cdot\nabla w+\gamma w, v)_T-\sum_{e\in\cE^i_h\cup\cE^{b,-}_h}(\bn\cdot\bz[w],\{v\})_e\\
            &\lesssim\left(\sum_{T\in\mathcal{T}_h}\|\nabla w\|^2_{L_2(T)}\right)^\frac12\|v\|_\LT+\|w\|_\LT\|v\|_\LT\\
            &\hspace{0.3cm}+\left(\sum_{e\in\cE^i_h\cup\cE^{b,-}_h}\frac{\sigma}{h_e}\|[w]\|^2_{L_2(e)}\right)^\frac12\left(\sum_{e\in\cE^i_h\cup\cE^{b,-}_h}\frac{h_e}{\sigma}\|\{v\}\|^2_{L_2(e)}\right)^\frac12\\
            &\lesssim\trinorm{w}_{h}\trinorm{v}_{h},
        \end{aligned}
    \end{equation*}
    where we use $\bz\in [W^{1,\infty}(\Omega)]^2$, $\gamma\in L_{\infty}(\Omega)$, \eqref{eq:traceinq} and \eqref{eq:dgpoin}.
    Furthermore, upon integration by parts, we have, for all $v\in V_h$,
    \begin{align*}
        \aarh(v,v)&=\sum_{T\in\mathcal{T}_h}(\bz\cdot\nabla v+\gamma v, v)_T-\sum_{e\in\cE^i_h\cup\cE^{b,-}_h}(\bn\cdot\bz[v],\{v\})_e\\
        &=\sum_{T\in\mathcal{T}_h}((\gamma-\frac12\nabla\cdot\bz)v,v)_T+\sum_{T\in\mathcal{T}_h}\int_{\partial T}\frac12(\bz\cdot\bn)v^2\ \!ds\\
        &\hspace{0.5cm}-\sum_{e\in\cE^i_h\cup\cE^{b,-}_h}\int_e\bz\cdot\bn[v]\{v\}\ \!ds\\
        &=\sum_{T\in\mathcal{T}_h}((\gamma-\frac12\nabla\cdot\bz)v,v)_T+\int_{\partial\O}\frac12|\bz\cdot\bn|v^2\ \!ds.
    \end{align*}
    By the assumption \eqref{eq:advassump}, we immediately have $\aarh(v,v)\ge0$. This finishes the proof.
    \end{proof}

\section{A Proof of \eqref{eq:dualityidentity}}\label{apdix:duality}

\begin{proof}
It follows from \eqref{eq:duality} that
\begin{equation}\label{eq:duall2eq}
    \begin{aligned}
        \|p&-p_h\|^2_{\LT}+\|y-y_h\|^2_{\LT}\\
        &=(\beta^{\frac{1}{2}}(-\Delta\xi+\bz\cdot\nabla\xi+\gamma\xi)-\theta,p-p_h)_\LT\\
        &\quad+(-\xi+\beta^{\frac{1}{2}}(\Delta\theta+\bz\cdot\nabla\theta-(\gamma-\nabla\cdot\bz)\theta),y-y_h)_\LT\\
        &=\beta^\frac12(-\Delta\xi,p-p_h)_\LT+\beta^\frac12\sum_{T\in\cT_h}(\bz\cdot\nabla\xi+\gamma\xi,p-p_h)_T\\
        &\quad-(\theta,p-p_h)_\LT-(\xi,y-y_h)_\LT\\
        &\quad+\beta^\frac12(\Delta\theta,y-y_h)_\LT+\beta^\frac12\sum_{T\in\cT_h}(\bz\cdot\nabla\theta-(\gamma-\nabla\cdot\bz)\theta,y-y_h)_T.
    \end{aligned}
\end{equation}
By the consistency of the SIP method (cf. \cite{riviere2008discontinuous,arnold2002unified}), we have
\begin{equation}\label{eq:sipdualconsis}
    (-\Delta\xi,p-p_h)=a_h^{sip}(\xi,p-p_h)\quad\text{and}\quad(\Delta\theta,y-y_h)=-a_h^{sip}(y-y_h,\theta).
\end{equation}
For the last term in \eqref{eq:duall2eq}, it follows from integration by parts that
\begin{equation}\label{eq:ardis1}
    \begin{aligned}
        &\sum_{T\in\cT_h}(\bz\cdot\nabla\theta-(\gamma-\nabla\cdot\bz)\theta,y-y_h)_T\\
        =&\sum_{T\in\cT_h}(-\bz\cdot\nabla(y-y_h),\theta)_T-(\gamma(y-y_h),\theta)_T\\
        &\quad+\sum_{T\in\cT_h}\int_{\partial T}(\bz\cdot\bn)(y-y_h)\theta\ \!ds.
    \end{aligned}
\end{equation}
The last term in \eqref{eq:ardis1} can be rewritten as the following \cite{arnold2002unified,di2011mathematical},
\begin{equation}\label{eq:ardis2}
    \begin{aligned}
        &\sum_{T\in\cT_h}\int_{\partial T}(\bz\cdot\bn)(y-y_h)\theta\ \!ds\\
        =&\sum_{e\in\cE_h^i}\int_e\bz\cdot\bn [(y-y_h)\theta]\ \!ds+\sum_{e\in\cE_h^b}\int_e\bz\cdot\bn(y-y_h)\theta\ \!ds\\
        =&\sum_{e\in\cE_h^i}\int_e\bz\cdot\bn [y-y_h]\{\theta\}\ \!ds+\sum_{e\in\cE_h^i}\int_e\bz\cdot\bn \{y-y_h\}[\theta]\ \!ds\\
        &+\sum_{e\in\cE_h^b}\int_e\bz\cdot\bn(y-y_h)\theta\ \!ds.
    \end{aligned}
\end{equation}
It then follows from $[\theta]=0$ on interior edges and $\theta=0$ on $\partial\O$ that
\begin{equation}\label{eq:ardis3}
    \sum_{T\in\cT_h}\int_{\partial T}(\bz\cdot\bn)(y-y_h)\theta\ \!ds=\sum_{e\in\cE_h^i\cup\cE_h^{b,-}}\int_e\bz\cdot\bn [y-y_h]\{\theta\}\ \!ds.
\end{equation}
According to \eqref{eq:ardis1}-\eqref{eq:ardis3}, we conclude
\begin{equation}\label{eq:ardualconsis}
    \sum_{T\in\cT_h}(\bz\cdot\nabla\theta-(\gamma-\nabla\cdot\bz)\theta,y-y_h)_T=-a_h^{ar}(y-y_h,\theta).
\end{equation}
Similarly, we can show 
\begin{equation}\label{eq:ardualconsis1}
    \sum_{T\in\cT_h}(\bz\cdot\nabla\xi+\gamma\xi,p-p_h)_T=a_h^{ar}(\xi,p-p_h).
\end{equation}
Therefore, we obtain the following by \eqref{eq:sipdualconsis}, \eqref{eq:ardualconsis}, \eqref{eq:ardualconsis1}, \eqref{eq:ahdef} and \eqref{eq:dgbilinear},
\begin{equation}
    \begin{aligned}
        &\|p-p_h\|^2_{\LT}+\|y-y_h\|^2_{\LT}\\
        &\quad=\beta^\frac12a_h(\xi,p-p_h)-(\theta,p-p_h)_\LT-(\xi,y-y_h)_\LT-\beta^\frac12a_h(y-y_h,\theta)\\
        &\quad=\cB_h((p-p_h,y-y_h),(\xi,\theta)).
    \end{aligned}
\end{equation}
\end{proof}

\section*{Acknowledgement}
The author would like to thank Prof. Susanne C. Brenner and Prof. Li-Yeng Sung for the helpful discussion regarding this project.

\bibliographystyle{plain}
\bibliography{references}

\end{document}